\documentclass[psamsfonts]{amsproc}

\usepackage{mathtools}
\usepackage{dsfont}
\usepackage{upref}
\usepackage{amssymb}
\usepackage[all]{xy}
\newtheorem{theorem}{Theorem}[section]
\newtheorem{lemma}[theorem]{Lemma}
\newtheorem{proposition}[theorem]{Proposition}
\newtheorem{definition-proposition}[theorem]{Definition-Proposition}

\theoremstyle{definition}
\newtheorem{definition}[theorem]{Definition}
\newtheorem{example}[theorem]{Example}

\theoremstyle{remark}
\newtheorem{remark}[theorem]{Remark}

\numberwithin{equation}{section}

\newcommand{\QQ}{\mathcal Q}
\newcommand{\PP}{\mathcal P}
\newcommand{\XX}{\mathcal X}

\newcommand{\cO}{\mathcal{O}}

\newcommand{\qn}{\textbf{q}}

\newcommand{\bk}{\overrightarrow{\textbf{k}}}
\newcommand{\bL}{\overrightarrow{\textbf{L}}}

\newcommand{\bm}{\overrightarrow{\textbf{m}}}
\newcommand{\br}{\overrightarrow{\textbf{r}}}

\newcommand{\bQ}{\overrightarrow{\textbf{Q}}}
\newcommand{\bI}{\overrightarrow{\textbf{I}}}
\newcommand{\bJ}{\overrightarrow{\textbf{J}}}
\newcommand{\bb}{\overrightarrow{\beta}}

\newcommand{\bS}{\overrightarrow{\textbf{S}}}

\newcommand{\Zn}{\mathbb Z^n_2}

\begin{document}

\title{On the Construction of $\Zn-$Supergrassmannians as Homogeneous $\Zn-$Superspaces\footnote{2010 Mathematics Subject Classification. Primary 58A50; Secondary 20N99. }}

\author[Mohammadi M.]{Mohammad Mohammadi}
\address{Department of Mathematics, Institute for Advanced Studies in Basic Sciences (IASBS), No. 444, Prof. Yousef Sobouti Blvd.
P. O. Box 45195-1159 Zanjan Iran, Postal Code 45137-66731}
\curraddr{}
\email{moh.mohamady@iasbs.ac.ir}
\thanks{}

\author[Varsaie S.]{Saad Varsaie}
\address{Department of Mathematics, Institute for Advanced Studies in Basic Sciences (IASBS), No. 444, Prof. Yousef Sobouti Blvd.
P. O. Box 45195-1159 Zanjan Iran, Postal Code 45137-66731}
\curraddr{}
\email{varsaie@iasbs.ac.ir}
\thanks{}

\date{}

\begin{abstract}
In this paper, we construct the $\Zn-$supergrassmannians by gluing of the $\Zn-$superdomains and give an explicit description of the action of the $\Zn-$super Lie group $GL(\bm)$ on the $\Zn-$supergrassmannian $G_{\bk}(\bm)$ in the functor of points language. In particular, we give a concrete proof of the transitively of this action, and the gluing of the local charts of the supergrassmannian.\\
\textbf{Keywords:} $\Zn-$Super Lie group, Homogeneous superspace, $\Zn-$ Supermanifold, $\Zn-$ Supergrassmannian 
\end{abstract}

\maketitle

\section*{Introduction}
There is growing interest in studying generalized supergeometry, that is, 
geometry of graded manifolds where the grading group is not $\mathbb Z_2$,
 but $\Zn=\mathbb Z_2\times \ldots \times \mathbb Z_2$. The foundational aspects of the theory of $\Zn-$supermanifolds were recently studied 
in \cite{supergeometry I}, \cite{spliting}, \cite{categoryz2n} and
\cite{Highertrace}. This generalization is used in string theory and parastatistics in physics,
 see \cite{phys1}, \cite{phys2}. Also in Mathematics, there exist many examples of $\Zn-$graded
$\Zn-$commutative algebras: quaternions and Clifford algebras, the algebra 
of Deligne differential superforms, etc. Moreover, there exist interesting  examples of $\Zn-$supermanifolds. In this paper, we study the $\Zn-$supergrassmannians as $\Zn-$supermanifolds and their constructions.

 In the context of supermanifolds, homogeneous superspaces have been defined and investigated extensively using the functor of points approach in \cite{g-spaces}, \cite{qoutiontsuper} and \cite{carmelibook}. In this paper, we show that $\Zn-$supergrassmannians $G_{\bk}(\bm)$ are  homogeneous, c.f. section \ref{supergrass}. To this end, we show that the $\Zn-$super Lie group $GL(\bm)$, c.f. section \ref{prelim}, acts transitively on $\Zn-$supergrassmannian $G_{\bk}(\bm)$ , c.f. section \ref{supergrass}. 

In the first section, we recall briefly all necessary basic concepts such as
$\Zn-$ grading spaces, $\Zn-$supermanifolds, $\Zn-$super Lie groups and an action of a $\Zn-$ super Lie group on a $\Zn-$supermanifold. we use these concepts in the case of $\mathbb Z_2-$supergeometry in \cite{g-spaces} and \cite{carmelibook}.  

\medskip

In section 2, we study the $\Zn-$supergrassmannians extensively. The $\mathbb Z_2-$ supergrassmannians are introduced by Manin in \cite{ma1}, but here by developing an efficient formalism, we fill in the details of the proof of this statement.

\medskip

In section 3, by a functor of points approach, an action of the super Lie group $GL(\bm)$ on the supergrassmannian
$G_{\bk}(\bm)$ is defined by gluing local actions. Finally it is shown that this action is transitive.

\section{Preliminaries} \label{prelim}
Let $\Zn=\mathbb Z_2\times \ldots \times \mathbb Z_2$ be the $n-$fold Cartesian product of $\mathbb Z_2.$ From now on, we set $\qn:=2^n-1$ and by $\bk$ and $\bm,$ we mean $(k_0,k_1,\ldots,k_\qn)$ and $(m_0,m_1,\dots,m_\qn)$ respectively such that $k_i,m_j\in \mathbb N$. Consider the bi-additive map 
\begin{align}\label{commutfactor}
\langle \ldotp, \ldotp \rangle & : \Zn \times \Zn \rightarrow \mathbb Z_2\nonumber\\
\langle a, b\rangle & =\sum_{i=1}^{n}a_ib_i (mod 2).\end{align} 
The even subgroup $(\Zn)_0$ consists of elements $\gamma \in \Zn$ such that $\langle \gamma, \gamma\rangle=0,$ and the set $(\Zn)_1$ consists of odd elements $\gamma \in \Zn$ such that $\langle \gamma, \gamma\rangle=1.$


One can fix an ordering on $\Zn$; based on this ordering, each even element is smaller than each odd element. Given two even (odd) elements $(a_1,a_2,\ldots,a_n)$ and $(b_1,b_2,\ldots,b_n)$, the first one is smaller than the second one for the lexicographical 
order, if $a_i<b_i$,  for the first i where $a_i$ and $b_i$ differ. For example, the lexicographical ordering on $\mathbb Z_2^3$ is
$$(0,0,0) < (0,1,1) < (1,0,1) < (1,1,0) < (0,0,1) < (0,1,0) < (1,0,0) < (1,1,1)$$
Obviously, $\Zn$ with lexicographical ordering is totally
ordered set. Thus it may be diagrammed as an ascending chain as follows $$\gamma_0 < \gamma_1 < \ldots < \gamma_{\qn}.$$

In the supergeometry, the sign rules between generators of the algebra are completely determined by their parity. One can define a grading by (\ref{commutfactor}) such that
$\epsilon(a,b)=(-1)^{\langle a,b\rangle}$ will be a sign rule what will lead to $\Zn-$supergeometry. Also, it has been shown that any other sign rule for finite number of coordinates is obtained from the above sine rule for sufficiently big $n.$ See \cite{supergeometry I} for more details.
\subsection{$\Zn-$supergeometry}
The $\Zn-$graded objects like $\Zn-$superlagebras, $\Zn-$super ringed spaces, $\Zn-$ superdomains and $\Zn-$supermanifolds have been studied in \cite{supergeometry I}, \cite{categoryz2n} and \cite{Highertrace}. In the following, we recall the necessary definitions from these references. 

By definition, a $\Zn-$super vector space is a direct sum $V=\bigoplus_{\gamma \in \Zn}V_{\gamma}$ of vector spaces $V_{\gamma}$ over a field $\mathbb K$ (with characteristic $0$). For each $\gamma\in \Zn,$ the elements of $V_{\gamma}$ is called homogeneous with degree $\gamma.$ If $x\in V_{\gamma}$
be a homogeneous element of $V,$ then the degree of $x$
is represented by $\tilde x=\gamma.$

A $\Zn-$superring $\mathcal R=\bigoplus_{\gamma \in \Zn}\mathcal R_{\gamma}$ is a ring such that its multiplication satisfy $\mathcal R_{\gamma_1}\mathcal R_{\gamma_2}\subset \mathcal R_{\gamma_1+\gamma_2}.$
A $\Zn-$superring $\mathcal R$ is called $\Zn-$commutative, if for any homogeneous elements $a,b\in \mathcal R$
$$a.b=(-1)^{\langle\tilde a, \tilde b\rangle}b.a$$
\begin{example}
Let $R$ be a ring and $\xi_1, \ldots, \xi_{\qn}$ be indeterminates with degree $\gamma_1, \ldots, \gamma_{\qn}\in \Zn$ respectively such that 
$$\xi_i\xi_j=(-1)^{\langle\gamma_i, \gamma_j\rangle}\xi_j\xi_i.$$
Then $R[[\xi_1, \ldots, \xi_{\qn}]]$ is the $\Zn-$commutative
associative unital $R-$algebra of formal series in the $\xi_a$ with coefficients in $R.$
\end{example}
By a $\Zn-$super ringed space, we mean a pair $(X, \cO_X)$ where $X$ is a topological space and $\mathcal O_X$ is a sheaf of $\Zn-$commutative $\Zn-$graded rings on $X$. A morphism between 
two $\Zn-$super ringed spaces 
$(X, \cO_X)$ and
$(Y, \cO_Y)$ is a pair $\psi:=(\overline{\psi},\psi^*)$ such that $\overline{\psi}:X\rightarrow Y$ is a continuous map and $\psi^*:\mathcal O_Y\rightarrow \overline{\psi}_*\mathcal O_X$ is a
homomorphism of weight zero between the sheaves of $\Zn-$commutative $\Zn-$graded rings.
Let $\qn=2^n-1,$ the $\Zn$-ringed space $$\mathbb R^{\bm}:=\Big(\mathbb R^{m_0},C^\infty_{\mathbb R^{m_0}}(-)[[\xi_1^1,\ldots,\xi_1^{m_1},\xi_2^1,\ldots,\xi_2^{m_2},\ldots,\xi_3^1,\ldots,\xi_\qn^{m_\qn}]]\Big)$$ is called $\Zn$-superdomain such that $C^\infty_{\mathbb R^{m_0}}$ is the sheaf of smooth functions on $\mathbb R^{m_0}$. Also for each open $U\subset \mathbb R^{m_0}$, 
$$C^\infty_{\mathbb R^{m_0}}(U)[[\xi_1^1,\ldots,\xi_1^{m_1},\xi_2^1,\ldots,\xi_2^{m_2},\ldots,\xi_3^1,\ldots,\xi_\qn^{m_\qn}]],$$
is the $\Zn-$commutative associative unital $\Zn-$superalgebra of formal power series
in formal variables $\xi_i^j$'s
of degrees $\gamma_i$ which commuting as follows:
$$\xi_i^j\xi_k^l=(-1)^{\langle\gamma_i,\gamma_k\rangle}\xi_k^l\xi_i^j.$$
By evaluation of $f=\sum_{}f_I\xi^I$ at $x \in U$, denoted by $ev_x(f)$, we mean $f_\phi(x)$.

A $\Zn$-supermanifold of dimension $\bm$ is a $\Zn$-ringed space $(\overline {M},\cO_M)$ that is locally isomorphic to $\mathbb R^{\bm}$. In addition $\overline {M}$ is a second countable and Hausdorff topological space. A morphism between two $\Zn$-supermanifolds $M=(\overline {M},\cO_M)$ and $N=(\overline{N},\cO_N)$ is a local morphism between two local $\Zn$-ringed spaces.

Analogous with supergeometry, one can obtain a $\Zn-$supermanifold by gluing $\Zn-$superdomains. We will use this method to construct the $\Zn-$ supergrassmannian as a $\Zn-$supermanifold in section \ref{supergrass}.
\subsection{Category theory}
By a locally small category, we mean a category such that the collection of all morphisms between any two of its objects is a set. Let $X$, $Y$ are objects in a category and $\alpha,\beta:X\rightarrow Y$ are morphisms between these objects. An universal pair $(E,\epsilon)$ is called \textit{equalizer} if the following diagram commutes:
$$E\xrightarrow{\epsilon} X\overset{\alpha}{\underset{\beta}{\rightrightarrows}}Y,$$
i.e., $\alpha \circ \epsilon=\beta \circ \epsilon$ and also for each object $T$ and any morphism $\tau:T\rightarrow X$ which satisfy $\alpha \circ \tau=\beta \circ \tau$, there exists unique morphism $\sigma:T\rightarrow E$ such that $\epsilon \circ \sigma = \tau$. If equalizer existed then it is unique up to isomorphism. For example, in the category of sets, which is denoted by $\textbf{SET}$, the equalizer of two morphisms $\alpha,\beta:X\rightarrow Y$ is the set $E=\{x\in X | \alpha(x)=\beta(x)\}$ together with the inclusion map  $\epsilon:E\hookrightarrow X.$

\medskip

Let $\mathcal{C}$ be a locally small category, and $X$ be an object in $\mathcal{C}$. By $T$-points of $X$, we mean $X(T):=Hom_\mathcal{C}(T,X)$ for any $T\in Obj(\mathcal{C})$. The functor of points of $X$ is a functor which is denoted by $X(.)$ and is defined as follows:

$$\begin{matrix} X(.):\mathcal{C} \rightarrow \textbf{SET}\\
\qquad\quad S \mapsto X(S),\\
\\
X(.): Hom_{\mathcal{C}}(S,T)\rightarrow Hom_{\textbf{SET}}(X(T),X(S))\\
\varphi \mapsto X(\varphi),\end{matrix}$$\\
where $X(\varphi):f\mapsto f\circ \varphi.$ A functor $F:\mathcal{C} \to \textbf{SET}$ is called representable if there exists an object $X$ in $\mathcal{C}$ such that $F$ and $X(.)$ are isomorphic. Then one may say that $F$ is represented by $X$. 
The category of functors from $\mathcal{C}$ to $\textbf{SET}$ is denoted by $[\mathcal{C}, \textbf{SET}]$. It is shown that the category of all representable functors from $\mathcal{C}$ to $\textbf{SET}$ is a subcategory of $[\mathcal{C}, \textbf{SET}]$.

\medskip

Corresponding to each morphism $\psi:X\rightarrow Y$, there exists a natural transformation
$\psi(.)$ from $X(.)$ to $Y(.)$. This transformation corresponds the mapping 
$\psi(T):X(T)\rightarrow Y(T)$ with $\xi\mapsto \psi\circ \xi$ for each $T\in Obj(\mathcal{C})$. Now set:
$$\begin{matrix} \mathcal{Y}:\mathcal{C}\rightarrow [\mathcal C,\textbf{SET}]\\ X \mapsto X(.)\\ \psi \mapsto \psi(.). \end{matrix}$$
Obviously, $\mathcal{Y}$ is a covariant  functor and it is called \textit{\textbf{Yoneda embedding}}.
\begin{lemma}\label{Yoneda}
The Yoneda embedding is full and faithful functor, i.e. the map $$Hom_\mathcal{C}(X,Y)\longrightarrow Hom_{[\mathcal{C},\textbf{SET}]}(X(.),Y(.)),$$ 
is a bijection for each $X,Y\in Obj(\mathcal{C}).$
\end{lemma}
\begin{proof}
see \cite{carmelibook}.
\end{proof}
Thus according to this lemma, $X,Y\in Obj(\mathcal{C})$ are isomorphic if and only if their functor of points are isomorphic. The Yoneda embedding is an equivalence between $\mathcal{C}$ and a subcategory of representable functors in $[\mathcal{C},\textbf{SET}]$ since not all functors are representable.

\subsection{$\Zn-$super Lie groups}
Let $\textbf{ZSM}$ be the category of $\Zn-$ supermanifolds. This is a category whose objects are $\Zn-$supermanifolds whose morphisms are morphisms between two $\Zn-$supermanifolds. Obviously, $\textbf{ZSM}$ is a locally small category and has finite product property. In addition it has a terminal object $\mathbb{R}^{\overrightarrow{\textbf{0}}}$, that is the constant sheaf $\mathbb{R}$ on a singleton $\{0\}$.

\medskip

Let $M=(\overline {M},\cO_M)$ be a $\Zn-$supermanifold and $p\in \overline {M}$. There is a map $j_p=(\overline{j}_p,{j_p}^*)$ where:
$$\begin{matrix} \overline{j}_p:\{0\} \rightarrow \overline {M},\qquad\qquad 
& {j_p}^*:\mathcal{O}_M\rightarrow \mathbb{R} \\
& \qquad\qquad\qquad\quad g \mapsto \tilde g(p)=:ev_p(g).\end{matrix}$$
So, for each $\Zn-$supermanifold $T$, one can define the morphism
\begin{align}\hat{p}_{_{_T}}:T\rightarrow \mathbb{R}^{\overrightarrow{\textbf{0}}}\xrightarrow{j_p}M\label{phat},\end{align}
as a composition of $j_p$ and the unique morphism $T\rightarrow \mathbb{R}^{\overrightarrow{\textbf{0}}}$.

\medskip

By $\Zn-$super Lie group, we mean a group-object in the category $\textbf{ZSM}$. More precisely, it is defined as follows:
\begin{definition}
A $\Zn-$super Lie group $G$ is a $\Zn-$supermanifold $G$ together with the morphisms of weight zero $\mu:G\times G \rightarrow G,\quad i:G \rightarrow G,\quad e:\mathbb{R}^{\overrightarrow{\textbf{0}}} \rightarrow G$ called multiplication, inverse and unit morphisms respectively, such that the following equations are satisfied
\begin{align*}
\mu \circ (\mu\times 1_G)&=\mu \circ (1_G\times \mu)\\
\mu \circ (1_G\times \hat{e}_G)\circ \bigtriangleup_G&=1_G=\mu \circ (\hat{e}_G\times 1_G)\circ \bigtriangleup_G\\
\mu \circ (1_G\times i)\circ \bigtriangleup_G&=\hat{e}_G=\mu \circ (i\times 1_G)\circ \bigtriangleup_G
\end{align*}
where $1_G$ is identity on $G$ and $\hat{e}_G$ is the composition of $e$ and the unique morphism $G\to R^{\overrightarrow{\textbf{0}}}$. In addition $\Delta_G$ is the diagonal map on $G$.
\end{definition}
Note that, there is a Lie group associated with each $\Zn-$super Lie group. Indeed, let $G$ be a $\Zn-$super Lie group and $G_0$ is reduced manifold associated to $G$ and $\mu_0, i_0, e_0$ are reduced morphisms associated to $\mu, i, e$ respectively. Since $G \rightarrow G_0$ is a functor, $(G_0, \mu_0, i_0, e_0)$ is a group-object of the category of differentiable manifolds.

\medskip
\begin{remark}\label{remark2}
Simply, one can show that any $\Zn-$super Lie group $G$ induced a group structure over its $T$-points for any arbitrary $\Zn-$supermanifold $T$. This means that the functor $T\rightarrow G(T)$ takes values in category of groups. Moreover, for any other $\Zn-$supermanifold $S$ and morphism $T\rightarrow S$, the corresponding map $G(S)\rightarrow G(T)$ is a homomorphism of groups. One can also define a $\Zn-$super Lie group as a representable functor $T\rightarrow G(T)$ from category $\textbf{ZSM}$ to category of groups. If such functor represented by a $\Zn-$supermanifold $G$, then the maps $\mu,i,e$ are obtained by Yoneda's lemma and the maps $\mu_{_T}:G(T)\times G(T)\rightarrow G(T),\quad i_{_T}:G(T)\rightarrow G(T)$ and $e_{_T}:\mathbb{R}^{\overrightarrow{\textbf{0}}}(T)\rightarrow G(T)$.
\end{remark}
\begin{remark}\label{charttheorem}
Consider the $\Zn-$superdomain $\mathbb R^{\bm}$ and an arbitrary $\Zn-$ supermanifold $T$. Let $f^j_i\in\cO(T)_{\gamma_i}$, $0\leq i \leq \qn$, $1\leq j \leq m_j$, be $\gamma_i-$degree elements. By Theorem 6.8 in \cite{supergeometry I} (Fundamental theorem of $\Zn$-morphisms), One may define a unique morphism $\psi:T\rightarrow \mathbb R^{\bm}$, 
by setting $\xi_i^j\mapsto f_i^j$ where $(\xi_i^j)$ is a  global coordinates system on $\mathbb R^{\bm}$. Thus $\psi$ may be represented by $(f_i^j)$.
\end{remark}
\begin{example}
Let $(t, \xi, \eta, \nu)$ be a global coordinates system on $\mathbb Z_2^2-$ superdomain $\mathbb{R}^{1|1|1|1}$. Let $T$ be an arbitrary supermanifold, we define:
\begin{align*}
   \mu_{_T}:\mathbb{R}^{1|1|1|1}(T)\times \mathbb{R}^{1|1|1|1}(T)& \longrightarrow  \mathbb{R}^{1|1|1|1}(T)
    \\
          (g_0,g_1,g_2,g_3),(g_0^{\prime},g_1^{\prime}, g_2^{\prime}, g_3^{\prime}) & \longmapsto (g_0+g_0^{\prime}, g_1+g_1^{\prime}, g_2+g_2^{\prime}, g_3+g_3^{\prime}),
\end{align*}
where
$g_i,g_i^{\prime}\in\cO(T)_{\gamma_i}$, $i=0,1,2,3.$ 
It follows that $\mathbb{R}^{1|1|1|1}(T)$ with $\mu_{_T}$ is a common group. Thus, by Remark \ref{remark2}, $\mathbb{R}^{1|1|1|1}$ is a $\mathbb Z_2^2-$super Lie group.  
\end{example} 
Analogously, one may show that the  $\Zn-$supersuperdomain $\mathbb{R}^{\bm}$ is a $\Zn-$super Lie group. 
\begin{example}
Let $V$ be a finite dimensional $\Zn$-super vector space of dimension $m_0|m_1|\ldots|m_\qn$ and let $\{R_1,\ldots, R_{m_0}, R_{m_0+1},\ldots, R_{m_0+m_1}, \ldots, R_{m_0+\ldots+m_\qn}\}$ be a  basis of $V$ for which the elements 
$R_{m_0+ ...+m_{i-1}+ k},$ $1\leq k\leq m_i,$
are of weight $\gamma_i$ for $0\leq i\leq\qn$.
Consider the functor $F$ from the category $\textbf{ZSM}$ to $\textbf{GRP}$ the category of groups which 
maps each $\Zn-$seupermanifold $T$ to $Aut_{\mathcal O(T)}(\mathcal O(T)\otimes V)$ the group of zero weight automorphisms of $\mathcal O(T)\otimes V$. Consider the $\Zn-$supermanifold $\textbf{End}(V)=\Big(\prod_{i} End(V_i),\mathcal{A}\Big)$ where $\mathcal{A}$ is the following sheaf 
\begin{align}\label{sheafEnd}
C^{\infty}_{\mathbb{R}^{m_0^2+ \ldots + m_\qn^2}}[[\xi_1^1,\ldots,\xi_1^{t_1},\xi_2^1,\ldots,\xi_2^{t_2},\ldots,\xi_3^1,\ldots,\xi_\qn^{t_\qn}]].
\end{align}
where $t_k=\sum_{\gamma_i+\gamma_j=\gamma_k}m_im_j.$ Let 
$F_{ij}$ be a linear transformation on $V$ defined  by $R_k\mapsto \delta_{ik}R_j$, then $\{F_{ij}\}$ is a basis for $\textbf{End}(V)$. If $\{f_{ij}\}$ is the corresponding dual basis, then it may be considered as a global coordinates on $\textbf{End}(V)$. Let $X$ be the open subsupermanifold of $\textbf{End}(V)$ corresponding to the open set:
$$\overline X=\prod_{i} GL(V_i)\subset \prod_{i} End(V_i).$$
Thus, we have
$$X=\Big(\prod_{i} GL(V_i),\mathcal{A}|_{\prod_{i} GL(V_i)}\Big).$$
It can be shown that the functor $F$ may be represented by $X$. For this, one may show that $Hom(T,X)\cong Aut_{\mathcal O(T)}(\mathcal O(T)\otimes V)$. To this end, first, note that
$$Hom(T,X)=Hom\Big(\mathcal A(X),\mathcal O(T)\Big).$$
 It is known that each $\psi\in Hom(\mathcal A(X),\mathcal O(T))$ may be uniquely determined by $\{g_{ij}\}$ where $g_{ij}=\psi(f_{ij})$, see \cite{vsv}.
 Now set $\Psi(R_j):=\Sigma g_{ij}R_i$.
 One may consider $\Psi$ as an element of $Aut_{\mathcal{O}(T)}(\mathcal{O}(T)\otimes V)$. Obviously $\psi\mapsto\Psi$ is a bijection from $Hom(T,X)$ to $Aut_{\mathcal{O}(T)}(\mathcal{O}(T)\otimes V)$.
Thus the $\Zn-$supermanifold $X$ is a $\Zn-$super Lie group and denoted it by $GL(V)$ or $GL(\bm)$ if $V=\mathbb{R}^{\bm}$.
Therefore $T$- points of $GL(\bm)$ are the $\bm\times \bm$ invertible $\Zn-$supermatrices of weight zero
\begin{equation*}
\left[
\begin{array}{c|c|c|c}
B_{00} &  B_{01} & \ldots & B_{0\qn}\\
\hline
\ldots &   &  & \ldots\\
\hline
B_{\qn0} &  B_{\qn1} & \ldots & B_{\qn\qn}\\
\end{array}
\right].
\end{equation*}
where the elements of the $m_k\times m_u$ block $B_{ku}$ have degree $\gamma_k+\gamma_u$ and the multiplication is the matrix product.
\end{example}
Let $x\in\overline{G}$, one can define the left and right translation by $x$ as 
\begin{align}
r_x:=\mu \circ (1_G\times \hat x_G)\circ \Delta_G,\label{pullefttrans}\\
l_x:=\mu \circ (\hat x_G\times 1_G)\circ \Delta_G,\label{pulrighttrans}
\end{align}
respectively. One can show that pullbacks of above morphisms are as following
\begin{align}
r_x^*:=(1_{\cO(G)}\otimes ev_x)\circ \mu^*,\label{lefttrans}\\
l_x^*:=(ev_x\otimes 1_{\cO(G)})\circ \mu^*.\label{righttrans}
\end{align}
One may also use the language of functor of points to describe two morphisms (\ref{pullefttrans}) and (\ref{pulrighttrans}).
\begin{definition}
Let $M$ be a $\Zn-$supermanifold and let $G$ be a $\Zn-$super Lie group with $\mu,i$ and $e$ as its multiplication, inverse and unit morphisms respectively. A morphism $a:M\times G\rightarrow M$ is called a (right) action of $G$ on $M$, if the following diagrams commute
\begin{Small}\begin{displaymath}
\xymatrix{
& M\times G\times G  \ar[rd]^{1_M\times\mu} \ar[ld]_{  a\times 1_G} & \\
M\times G \ar[dr]_a &  & M\times G, \ar[ld]^{a}\\
& M & }\quad
\xymatrix{
& M\times G \ar[rd]^{a} & \\
M\ar[ur]^{(1_M\times \hat{e}_{_M})\circ \Delta_M} \ar[rr]_{1_M} &  & M,}
\end{displaymath}\end{Small}
where $\hat{e}_{_M}, \Delta_M$ are as above. In this case, we say G acts from right on $M$. One can define left action analogously. 
\end{definition}
According to the above diagrams, one has:
$$\begin{matrix} a\circ(1_M\times\mu)=a\circ(a\times 1_G),&\qquad\qquad & a\circ(1_M\times\hat{e}_M)\circ \Delta_M=1_M. \end{matrix}$$
By Yoneda lemma (Lemma \ref{Yoneda}), one may consider, equivalently, the action of G as a natural transformation: 
$$a(.):M(.)\times G(.)\rightarrow M(.).$$
Thus for each supermanifold $T$, the morphism $a(T): M(T)\times G(T)\rightarrow  M(T)$ is an action of group $G(T)$ on the set $M(T)$. This means:
\begin{enumerate}
\item[1.] $(\PP.\QQ_1).\QQ_2=\PP.(\QQ_1\QQ_2),\qquad \forall \QQ_1,\QQ_2\in G(T), \forall \PP\in M(T).$
\item[2.] $\PP.\hat{e}_{_T}=\PP,\qquad \forall  \PP\in M(T).$
\end{enumerate}
Let $p\in\overline{M},$ define
\begin{align*} & a_p:G\rightarrow M, \qquad\qquad\qquad\qquad\qquad a^g:M\rightarrow M,\\ & a_p:=a\circ (\hat{p}_{_G}\times 1_{_G})\circ \Delta_{_G}, \qquad\qquad a^g:=a\circ (\hat{g}_{_M}\times 1_M)\circ \Delta_M, \end{align*}
where $\hat{p}_{_G}$ and $\hat{g}_{_M}$ are the morphism (\ref{phat}) for $p\in \overline M$ and $g\in \overline{G}$ respectively. Equivalently, these maps may be defined as
\begin{align}\label{apag} (a_p)_{_T}:& G(T)\rightarrow M(T), \qquad\qquad\qquad (a^g)_{_T}:M(T)\rightarrow M(T),\nonumber\\ & \QQ\longmapsto \hat{p}_{_T}.\QQ, \qquad\qquad\qquad\qquad\qquad\quad \PP\longmapsto \PP.\hat{g}_{_T}. \end{align}
One may easily show that $a_p$ has constant rank(see Proposition 8.1.5 in \cite{carmelibook}, for more details). Before next definition, we recall
that a morphism between $\Zn-$ supermanifolds, say $\psi:M\rightarrow N$ is a submersion at $x\in\overline {M}$, if $(d\psi)_x$ is surjective and $\psi$ is called submersion, if it is surjective at each point. (For more details, refer to \cite{vsv}, \cite{carmelibook}). Also $\psi$ is a \textit{surjective submersion}, if in addition $\psi_0$ is surjective.
\begin{definition}
 Let $G$ acts on $M$ with action $a:M\times G \rightarrow M$. The action $a$ is called transitive, if there exist $p\in \overline {M}$ such that $a_p$ is a surjective submersion.  
 \end{definition}
It is shown that, if $a_p$ is a submersion for one $p\in \overline {M}$, then it is a submersion for all point in $\overline {M}$.
The following proposition will be required in the last section.
\begin{proposition}\label{Transitive}
Let $a:M\times G \rightarrow M$ be an action. Then $a$ is transitive if and only if for a $p\in\overline{M},$ $(a_p)_{\mathbb{R}^{\br^\prime}}:G(\mathbb{R}^{\br^\prime})\rightarrow M(\mathbb{R}^{\br^\prime})$ is surjective, where $\br=(r_0,r_1,\ldots,r_\qn)$ is the dimension of $G$ and $\br^\prime:=(0,r_1,\ldots,r_\qn).$
\end{proposition}
\begin{proof} The proof is the same as the proof of proposition 9.1.4 in \cite{carmelibook} with appropriate modifications. 
\end{proof}
\begin{definition}
Let $G$ be a $\Zn-$super Lie group and let $a$ be an action of $G$ on $\Zn-$supermanifold $M$. By \textit{stabilizer} of $p\in \overline {M},$ we mean a $\Zn-$supermanifold $G_p$ equalizing the diagram $$G\overset{a_p}{\underset{\hat{p}_G} {\rightrightarrows}}M.$$
\end{definition}
\begin{proposition}\label{Isotropic}
Let $a:M\times G \rightarrow M$ be an action, then
\begin{enumerate} 
\item[1.] The  following diagram admits an equalizer $G_p$ 
$$G\overset{a_p}{\underset{\hat{p}_G} {\rightrightarrows}}M.$$
\item[2.] $G_p$ is a $\Zn-$subsuper Lie group of $G$.
\item[3.] The functor $T\rightarrow (G(T))_{\hat{p}_{_T}}$ is represented by $G_p$, where $(G(T))_{\hat{p}_{_T}}$ is the stabilizer in $\hat{p}_{_T}$ of the action of $G(T)$ on $M(T)$.
\end{enumerate} \end{proposition}
\begin{proof} The proof is the same as the proof of proposition 8.4.7 in \cite{carmelibook} with appropriate modifications. \end{proof}
We end this section with the next proposition which is a straightforward $\Zn$ generalization of the proposition 6.5 in \cite{g-spaces}. 
 \begin{proposition}\label{equivariant}
 Suppose $G$ acts transitively on $M$. There exists a $G$-equivariant isomorphism 
\begin{displaymath}\xymatrix{ \dfrac{G}{G_p}\ar[r]^{\cong} & M. } \end{displaymath} \end{proposition} 
\section{$\Zn$-supergrassmannian}\label{supergrass}
Supergrassmannians $G_{k|l}(m|n)$ are introduced by Manin in \cite{ma1} and the authors have studied them in more details in \cite{OurArticle} and \cite{MyArticle}. In this section, we introduce the $\mathbb Z_2^n$-supergrassmannian which is denoted by $G_{k_0|k_1|\ldots|k_\qn}(m_0|m_1|\ldots|m_\qn)$ or  
$G_{\bk}(\bm)$ in short.
For convenience from now, we set 
$$\begin{matrix}
&\beta_0:=\sum_{\gamma_i+\gamma_j=\gamma_0}k_i(m_j-k_j),\\
& \beta_1:=\sum_{\gamma_i+\gamma_j=\gamma_1}k_i(m_j-k_j)\\
&\ldots \\
& \beta_\qn:=\sum_{\gamma_i+\gamma_j=\gamma_\qn}k_i(m_j-k_j)\\
\end{matrix}$$
and also decompose any $\Zn-$supermatrix into $2^n\times 2^n$ blocks 
\begin{equation*}
\left[
\begin{array}{c|c|c|c}
B_{00} &  B_{01} & \ldots & B_{0\qn}\\
\hline
\ldots &   &  & \ldots\\
\hline
B_{\qn0} &  B_{\qn1} & \ldots & B_{\qn\qn}\\
\end{array}
\right].
\end{equation*}
such that the elements of block $B_{ku}$ have degree $\gamma_k+\gamma_u$.
By a $\Zn$-supergrassmannian, $G_{\bk}(\bm)$, we mean a
$\Zn$-supermanifold which is constructed by gluing $\Zn$- superdomains
$\mathbb{R}^{\bb}=\Big(\mathbb{R}^{\beta_0}, \, C^{\infty}_{\mathbb{R}^{\beta_0}}(-)[[\xi_1^1,\ldots,\xi_1^{\beta_1},\xi_2^1,\ldots,\xi_2^{\beta_2},\ldots,\xi_\qn^1,\ldots,\xi_\qn^{\beta_\qn}]]\Big)$ 
as follows:\\
For $i=0,1,\ldots,\qn,$ let $I_i \subset \{1, \ldots, m_i\}$ be a sorted subset in ascending order with 
$ k_i $ elements. The elements of $ I_i$ are called $\gamma_i$-degree indices. The multi-index $\bI=(I_1, ..., I_\qn)$ is called  $\bk$-index. Set $\mathcal U_{\bI}:=(\overline{U}_{\bI}, \mathcal{O}_{\bI})$, where
$$\overline{U}_{\bI}=\mathbb{R}^{\beta_0} \quad,\quad \mathcal{O}_{\bI}=\, C^{\infty}_{\mathbb{R}^{\beta_0}}(-)[[\xi_1^1,\ldots,\xi_1^{\beta_1},\xi_2^1,\ldots,\xi_2^{\beta_2},\ldots,\xi_\qn^1,\ldots,\xi_\qn^{\beta_\qn}]].$$
Let each $\Zn$-superdomain $\mathcal U_{\bI}$ be labeled by a $\Zn-$supermatrix $\bk\times \bm$ of weight zero, say $A_{\bI}$, with $2^n\times 2^n$ blocks $ B_{ij}$ each of which is a $k_i\times m_j$ matrix. In addition, except for columns with indices in $I_0 \cup I_1\cup \ldots \cup I_\qn$, which together form a minor denoted by $M_{\bI}A_{\bI}$, the matrix is filled from up to down and left to right by $x_a^{\bI}, \xi_b^{\bI}$, the free generators of $\mathcal{O}_{\bI}(\mathbb R^{\beta_0})$ each of them sits in a block with same degree. This process impose an ordering on the set of generators. In addition $ M_{\bI}A_{\bI} $ is supposed to be the identity matrix.

For example, consider $G_{1|2|1|1|}{2|2|2|2}$. Then let $I_0=\{1\}, I_1=\{1,2\}, I_2=\{1\}, I_3=\{2\}$, so $ \bI $ is a $1|2|1|1$-index. In this case the set of generators of $\mathcal{O}_{\bI}(\mathbb{R}^{\beta_0})$ is 
\begin{equation*} \{x^1, x^2, x^3, \xi_1^1,\xi_1^2,\xi_1^3,\xi_1^4,\xi_2^1, \xi_2^2, \xi_2^3, \xi_2^4, \xi_3^1, \xi_3^2, \xi_3^3, \xi_3^4\}, \end{equation*}
 and $ A_{\bI} $ is:
\begin{equation*}
\left[
\begin{array}{cc|cc|cc|cc}
1 & x^1 & 0 & 0 & 0 & \xi_2^2 & \xi_3^4 & 0 \\
\hline
0 & \xi_1^1 & 1 & 0 & 0 & \xi_3^2 & \xi_2^3 & 0 \\
0 & \xi_1^2 & 0 & 1 & 0 & \xi_3^3 & \xi_2^4 & 0 \\
\hline
0 & \xi_2^1 & 0 & 0 & 1 & x_2 & \xi_1^4 & 0 \\
\hline
0 & \xi_3^1 & 0 & 0 & 0 & \xi_1^3 & x^3 & 1 \\
\end{array}
\right].
\end{equation*}
Note that, in this example,
\begin{align}\label{ordergen}
\{x^1,\xi_1^1,\xi_1^2,\xi_2^1,\xi_3^1,\xi_2^2,\xi_3^2,\xi_3^3,x_2,\xi_1^3,\xi_3^4,\xi_2^3,\xi_2^4,\xi_1^4 ,x^3\}
\end{align}
is corresponding total ordered set of generators.

By $\widetilde{U}_{{\bI},{\bJ}},$ we mean the set
of all points of $\overline{U}_{\bI}$, on which $M_{\bJ}A_{\bI}$ is invertible.
Obviously $\widetilde{U}_{{\bI},{\bJ}}$ is an
open set.
 The transition map between the two $\Zn$-superdomains $U_{\bI}$ and $U_{\bJ}$ is denoted by 
$$g_{_{\bI,\bJ}}:\Big(\widetilde{U}_{{\bJ},{\bI}},\mathcal{O}_{\bJ}|_{\widetilde{U}_{{\bJ},{\bI}}}\Big)\longrightarrow \Big(\widetilde{U}_{{\bI},{\bJ}},\mathcal{O}_{\bI}|_{\widetilde{U}_{{\bI},{\bJ}}}\Big).$$
Note that $g_{_{\bI,\bJ}}=(\overline{g}_{_{\bI,\bJ}},g^*_{_{\bI,\bJ}}),$ where
$g^*_{_{\bI,\bJ}}$ is an isomorphism between sheaves determined by defining on each entry of $D_{\bI}(A_{\bI})$ as a rational expression
which appears as the corresponding entry provided by the pasting equation 
\begin{equation}\label{transitionmap}
D_{\bI}\bigg(\big(M_{\bI}A_{\bJ}\big)^{-1}A_{\bJ}\bigg)=D_{\bI}(A_{\bI}),\end{equation}
where $D_{\bI}(A_{\bI})$ is a matrix which is remained after omitting $M_{\bI}A_{\bI}$. Clearly, the left hand side of \eqref{transitionmap} is defined whenever $M_{\bI}A_{\bJ}$ is invertible. 
The morphism  $g^*_{_{\bI,\bJ}}$ induces the continuous map $\overline{g}_{_{\bI,\bJ}}$ (in the case $n=1$, see \cite{roshandel}, lemma 3.1).

For example in $G_{1|2|1|1}(2|2|2|2)$ suppose 
\begin{align*}
& I_0=\{1\}, I_1=\{1,2\}, I_2=\{1\}, I_3=\{2\},\\
& J_0=\{2\}, J_1=\{1,2\}, J_2=\{2\}, J_3=\{1\},
\end{align*}
so $\bI, \bJ$ are $ 1|2|1|1 $-indices. We have:
\begin{equation*}
A_{\bI}=\left[
\begin{array}{cc|cc|cc|cc}
1 & x^1 & 0 & 0 & 0 & \xi_2^2 & \xi_3^4 & 0 \\
\hline
0 & \xi_1^1 & 1 & 0 & 0 & \xi_3^2 & \xi_2^3 & 0  \\
0 & \xi_1^2 & 0 & 1 & 0 & \xi_3^3 & \xi_2^4 & 0 \\
\hline
0 & \xi_2^1 & 0 & 0 & 1 & x^2 & \xi_1^4 & 0  \\
\hline
0 & \xi_3^1 & 0 & 0 & 0 & \xi_1^3 & x^3 & 1  \\
\end{array}
\right], \quad 
A_{\bJ}=\left[
\begin{array}{cc|cc|cc|cc}
x^1 & 1 & 0 & 0 & \xi_2^2 & 0 & 0 & \xi_3^4 \\
\hline
\xi_1^1 & 0 & 1 & 0 & \xi_3^2 & 0 & 0 & \xi_2^3  \\
\xi_1^2 & 0 & 0 & 1 & \xi_3^3 & 0 & 0 & \xi_2^4 \\
\hline
\xi_2^1 & 0 & 0 & 0 & x^2 & 1 & 0 & \xi_1^4  \\
\hline
\xi_3^1 & 0 & 0 & 0 & \xi_1^3 & 0 & 1 & x^3  \
\end{array}\right],
\end{equation*}
\begin{equation*}
M_{\bJ}A_{\bI}=\left[
\begin{array}{c|cc|c|c}
x^1 & 0 & 0 & \xi_2^2 & \xi_3^4 \\
\hline
\xi_1^1 & 1 & 0 & \xi_3^2 & \xi_2^3 \\
\xi_1^2 & 0 & 1 & \xi_3^3 & \xi_2^4 \\
\hline
\xi_2^1 & 0 & 0 & x^2 & \xi_1^4 \\
\hline
\xi_3^1 & 0 & 0 & \xi_1^3 & x^3 \\
\end{array}
\right].
\end{equation*}
The maps $g_{_{\bI,\bJ}}$ are gluing morphisms. In fact, a straightforward computation shows the following proposition holds.
\begin{proposition} Let $g_{_{\bI,\bJ}}=\big(\overline{g}_{_{\bI,\bJ}},g^*_{_{\bI,\bJ}}\big)$ be as above, then
\begin{enumerate}
\item[1.] $g^*_{_{\bI,\bI}}=id.$
\item[2.] $g^*_{_{\bJ,\bI}} \circ g^*_{_{\bI,\bJ}}=id.$
\item[3.] $g^*_{_{\bS,\bI}} \circ g^*_{_{\bJ,\bS}}\circ g^*_{_{\bI,\bJ}}=id.$
\end{enumerate}
\end{proposition}
\begin{proof}
For first equality, note that the map $ g^*_{_{\bI,\bI}} $ is obtained from the following equality:
\begin{equation*}
D_{\bI}\bigg((M_{\bI}A_{\bI})^{-1}A_{\bI}\bigg)=D_{\bI}A_{\bI},
\end{equation*}
where the matrix $ M_{\bI}A_{\bI} $ is identity. So $g^*_{_{\bI,\bI}}$ is defined by the following equality:
\begin{equation*}
D_{\bI}A_{\bI}=D_{\bI}A_{\bI}.
\end{equation*}
This shows the first equality. For second equality, let $ \bJ $ be an another $ \bk $-index, so $ g^*_{_{\bJ,\bI}} $ is obtained by the following equality:
\begin{equation*}
D_{\bJ}\bigg((M_{\bJ}A_{\bI})^{-1}A_{\bI}\bigg)=D_{\bJ}A_{\bJ}.
\end{equation*}
One may see that $g^*_{_{\bJ,\bI}} \circ g^*_{_{\bI,\bJ}}$ is obtained by following equality:
\begin{equation*}
D_{\bI}\bigg(\bigg(M_{\bI}\Big((M_{\bJ}A_{\bI})^{-1}A_{\bI}\Big)\bigg)^{-1}(M_{\bJ}A_{\bI})^{-1}A_{\bI}\bigg)=D_{\bI}A_{\bI}.
\end{equation*}
For left side, we have
 \begin{align*}&=D_{\bI}\Bigg(\bigg((M_{\bJ}A_{\bI})^{-1}M_{\bI}A_{\bI}\bigg)^{-1}(M_{\bJ}A_{\bI})^{-1}A_{\bI}\Bigg)\\
 &=D_{\bI}\bigg(\Big((M_{\bJ}A_{\bI})^{-1}\Big)^{-1}(M_{\bJ}A_{\bI})^{-1}A_{\bI}\bigg)\\
 &=D_{\bI}\bigg((M_{\bJ}A_{\bI})(M_{\bJ}A_{\bI})^{-1}A_{\bI}\bigg)=D_{\bI}(A_{\bI}).\end{align*}
Accordingly the map $g^*_{_{\bJ,\bI}} \circ g^*_{_{\bI,\bJ}}$ is obtained by $ D_{\bI}A_{\bI}=D_{\bI}A_{\bI} $ and it shows that this map is identity.
For third equality, it is sufficient to show that the map $g^*_{_{\bS,\bI}} \circ g^*_{_{\bJ,\bS}}\circ g^*_{_{\bI,\bJ}}$ is obtained from
\begin{equation*}
 D_{\bI}A_{\bI}=D_{\bI}A_{\bI} .
\end{equation*}
This case obtains from case $2$ analogously. 
\end{proof}
So the sheaves $(\overline{U}_{\bI}, \mathcal O_{\bI})$ may be glued through the $g_{\bI, \bJ}$ to construct the $\Zn-$supergrassmannian $G_{\bk}(\bm)$. Indeed, according to \cite{vsv}, the conditions of the above proposition are necessary and sufficient for gluing.
\section {$\Zn-$Supergrassmannian as homogeneous $\Zn-$superspace}
Let $G=(\overline{G},\mathcal{O}_G)$ be a $\Zn-$super Lie group and  $H=(\overline{H},\mathcal{O}_H)$ be a closed $\Zn-$sub super Lie group of $G$.
One can define a $\Zn-$supermanifold structure on the topological space $\overline{X}=\overline{G}/\overline{H}$ as follows:\\
Let $\mathfrak{g}=Lie(G)$ and $\mathfrak{h}=Lie(H)$ be the $\Zn-$super Lie algebras corresponding with $G$ and $H$. For each $Z\in \mathfrak{g}$, let $D_{Z}$ be the left invariant vector field on $G$ associated with $Z$. For $\mathfrak{h},$ a $\Zn-$subalgebra of $\mathfrak{g},$ set:
$$\forall U\subset \overline{G} \qquad \mathcal{O}_{\mathfrak{h}}(U):=\{f\in \mathcal{O}_{G}(U)| D_{Z}f=0 \quad on\hspace{5pt} U, \quad \forall Z\in \mathfrak{h}\}.$$
On the other hand, for any open subset $U\subset \overline{G}$ set:
$$\mathcal{O}_{inv}(U):=\{f\in\mathcal{O}_{G}(U)|\quad \forall x_0 \in \overline{H},\hspace{5pt} r^*_{x_0}f=f \},$$
where $r_{x_0}$ is the right translation by $x_0$ in \eqref{pulrighttrans}.
If $\overline{H}$ is connected, then $\mathcal{O}_{inv}(U)= \mathcal{O}_{\mathfrak{h}}(U)$.
Let $\overline \pi:\overline G\rightarrow \overline X$ be the natural projection. For each open subset $W\subset \overline{X}=\overline{G}/\overline{H}$, the structure sheaf $\mathcal{O}_X$ is defined as following
$$\mathcal{O}_{X}(W):=\mathcal{O}_{inv}(U)\cap \mathcal{O}_{\mathfrak{h}}(U),$$ 
where $U=\overline{\pi}^{-1}(W).$
One can show that $\mathcal{O}_{X}$ is a sheaf on $\overline{X}$ and the ringed space $X=(\overline{X},\mathcal{O}_{X})$ is a $\Zn-$superdomain locally (See \cite{g-spaces} for more details). So $X$ is a $\Zn-$supermanifold and is called homogeneous $\Zn-$superspace.
In this section, we want to show that the $\Zn-$supergrassmannian $G_{\bk}(\bm)$ is a homogeneous $\Zn-$superspace. According to the section 1, it is enough to find a $\Zn-$super Lie group which acts on $G_{\bk}(\bm)$ transitively. For this, we need the following remark and the next lemma
\begin{remark}\label{morph-matrix}
Let $\XX$ be an element of $U_{\bI}(T)$ where $\bI$ is an arbitrary index. One can correspond to $\XX$ a $\Zn-$supermatrix $\bk\times \bm$ called $[\XX]_{\bI}$ as follows:
Except for columns with indices in $ I_0 \cup I_1 \cup \ldots \cup I_\qn $, the blocks are filled from up to down and left to right by $ f_i, g_j $'s where
$$f_i:=\XX(x_i), \quad g_j:=\XX(\xi_j),$$
according to the ordering (\ref{ordergen}), where $(x_i; \xi_j)$ is the global coordinates of the $\Zn-$superdomain $U_{\bI}$. The columns with indices in $ I_0 \cup I_1 \cup \ldots \cup I_\qn $ form an identity matrix.
\end{remark}
\begin{lemma}
Let $\psi:T\rightarrow \mathbb{R}^{\br}$ be a $T$-point of $\mathbb{R}^{\br}$ and $(z_{tu})$ be a global coordinates of  $\mathbb{R}^{\br}$ with ordering as the one introduced in (\ref{ordergen}). If $B=(\psi^*(z_{tu}))$ is the $\Zn-$supermatrix corresponding to $\psi$, then the $\Zn-$supermatrix corresponding to $\big(g_{_{\bI,\bJ}}\big)_{_T}(\psi)$ is as follows:
\begin{equation*}
 D_{\bI}\Big((M_{\bI}[B]_{\bJ})^{-1}[B]_{\bJ}\Big),
\end{equation*}
where $[B]_{\bJ}$ is introduced in Remark \ref{morph-matrix}.
\end{lemma}
\begin{proof}
Note that $g^*_{\bI, \bJ}$ may be represented by a $\Zn-$supermatrix as follows:
\begin{equation*}
 D_{\bI}\Big((M_{\bI}A_{\bJ})^{-1}A_{\bJ}\Big),
\end{equation*}
where $A_{\bJ}$ is the label of $U_{\bJ}.$
Let $M_{\bI}A_{\bJ}=(m_{tu})$ and $(M_{\bI}A_{\bJ})^{-1}=(m^{tu})$. If $z=(z_{ij})$ be a coordinates system on $U_{\bI}$, then one has
$$g^*_{\bI, \bJ}(z_{tu})=\sum m^{tk}(z).z_{ku}.$$
Then
\begin{align*}
\psi^*\circ g^*_{\bI, \bJ}(z_{tu})&= \psi^*\big(\sum m^{tk}(z).z_{ku}\big)=\sum m^{tk}\big(\psi^*(z)\big).\psi^*(z_{ku}).
\end{align*}
For second equality one may note that $\psi^*$ is a homomorphism of $\Zn-$superalgebras and $m^{tk}(z)$ is a rational function of $z$. 
Obviously, the last expression is the $(t, u)$-entry of the matrix $D_{\bI}\Big((M_{\bI}[B]_{\bJ})^{-1}[B]_{\bJ}\Big)$. This completes the proof.
\end{proof}
\begin{theorem}
The $\Zn-$super Lie group $GL(\bm)$ acts on $\Zn-$ supergrassmannian $G_{\bk}(\bm)$. 
\end{theorem}
\begin{proof}
First, we have to define a morphism 
$a:G_{\bk}(\bm)\times GL(\bm) \rightarrow G_{\bk}(\bm)$. For this, by Yoneda lemma, it is sufficient to define $a_{_T}$: $$a_{_T}:G_{\bk}(\bm)(T)\times GL(\bm)(T) \rightarrow G_{\bk}(\bm)(T).$$
for each $\Zn-$supermanifold $T$ or equivalently define 
$$(a_{_T})^\PP: G_{\bk}(\bm)(T) \rightarrow G_{\bk}(\bm)(T).$$
where $\PP$ is a fixed arbitrary element in $GL(\bm)(T)$. For brevity, we denote $(a_{_T})^\PP$ by $\textbf{A}$.
One may consider $GL(\bm)(T)$, as the set of $\bm\times \bm$ invertible $\Zn-$supermatrices with entries in $\mathcal{O}(T)$, but there is not such a description for $G_{\bk}(\bm)(T)$, because it is not a $\Zn-$superdomain. 
We know each $\Zn-$supergrassmanian is constructed by gluing $\Zn-$superdomains (c.f. section 2), so one may define the actions of $GL(\bm)$ on $\Zn-$superdomains $(\overline{U}_{\bI},\mathcal{O}_{\bI})$ and then shows that these actions glued to construct $a_{_T}$. 

For defining $\textbf{A}$, it is needed to refine the covering $\{U_{\bI}(T)\}_{\bI}$. Set
$$U_{\bI}^{\bJ}(T):=\Big\{\psi\in U_{\bI}(T)\quad|\quad M_{\bJ}\Big([\psi]_{\bI}[\PP]\Big) \qquad \text{is invertible} \Big\},$$
where $[\PP]$ is the matrix form of the fixed arbitrary element $\PP$ in $GL(\bm)(T)$, see \cite{carmelibook} and \cite{vsv}. One can show that $\{U_{\bI}^{\bJ}(T)\}_{\bI,\bJ}$ is a covering for $G_{\bk}(\bm)(T)$ and $\textbf{A}  \Big(U_{\bI}^{\bJ}(T)\Big)\subseteq U_{\bJ}(T).$
Now consider all maps 
\begin{align*} 
\textbf{A}_{\bI}^{\bJ}: & U_{\bI}^{\bJ}(T)\rightarrow U_{\bJ}(T)\\
& \quad \psi \mapsto  D_{\bJ}\bigg(\Big(M_{\bJ}([\psi]_{\bI}[\PP])\Big)^{-1}[\psi]_{\bI}[\PP]\bigg)\end{align*}
where, $[\psi]_{\bI}$ is as above. We have to show that these maps may be glued to construct a global map on $G_{\bk}(\bm)(T)$. For this, it is sufficient to show that
the following diagram commutes:
\begin{center}
\begin{displaymath}\label{Diag}
\xymatrix{ & U_{\bI}^{\bJ}(T)\cap U_{\bQ}^{\bL}(T) \ar[rd]^{(g_{{\bQ},{\bI}})_{_T}}\ar[ld]_{\textbf{A}_{\bI}^{\bJ}} & \\ U_{\bJ}(T)\cap U_{\bL}(T)\ar[dr]_{(g_{{\bJ},{\bL}})_{_T}} &  & U_{\bI}^{\bJ}(T)\cap U_{\bQ}^{\bL}(T) \ar[ld]^{\textbf{A}_{\bQ}^{\bL}}\\
& U_{\bJ}(T)\cap U_{\bL}(T) & }\quad 
\end{displaymath}
\end{center}
where $\big(g_{\bI, \bJ}\big)_{_T}$
is the induced map from $g_{\bI, \bJ}$ on $T$-points. The following proposition is used to show commutativity of the above diagram. 
\end{proof}
\begin{proposition}
The last diagram commutes.
\end{proposition}
\begin{proof}
We have to show that
\begin{equation}\label{glueaction}
(g_{{\bL},{\bJ}})_{_T}\circ \textbf{A}_{\bI}^{\bJ}=\textbf{A}_{\bQ}^
{\bL} \circ (g_{{\bQ},{\bI}})_{_T}.
\end{equation}
for arbitrary $\bk$-indices $\bI, \bJ, \bQ, \bL.$ 
Let $\psi\in U_{\bI}^{\bJ}(T)\cap U_{\bQ}^{\bL}(T)$ be an arbitrary element. One has $\psi\in U_{\bI}^{\bJ}(T)$, so 
\begin{align*} D_{\bJ}\bigg(\Big(M_{\bJ}([\psi]_{\bI}[\PP])\Big)^{-1}[\psi]_{\bI}[\PP]\bigg)&\in U_{\bJ}(T),\\
(g_{{\bL},{\bJ}})_{_T}\Bigg(D_{\bJ}\Big(\big(M_{\bJ}([\psi]_{\bI}[\PP])\big)^{-1}[\psi]_{\bI}[\PP]\Big)\Bigg)&\in U_{\bL}(T).
\end{align*}
From left side of (\ref{glueaction}), we have:
\begin{align*}
(g_{{\bL},{\bJ}})_{_T}&\circ \textbf{A}_{\bI}^{\bJ}(\psi)\\
&=(g_{{\bL},{\bJ}})_{_T}\Bigg(D_{\bJ}\Big(\big(M_{\bJ}([\psi]_{\bI}[\PP])\big)^{-1}[\psi]_{\bI}[\PP]\Big)\Bigg)\\
& =D_{\bL}\Bigg(\bigg(M_{\bL}\Big((M_{\bJ}([\psi]_{\bI}[\PP]))^{-1}[\psi]_{\bI}\PP\Big)\bigg)^{-1}(M_{\bJ}([\psi]_{\bI}\PP))^{-1}[\psi]_{\bI}[\PP]\Bigg)\\
& =D_{\bL}\Bigg(\Big((M_{\bJ}([\psi]_{\bI}[\PP]))^{-1}(M_{\bL}([\psi]_{\bI}[\PP]))\Big)^{-1}(M_{\bJ}([\psi]_{\bI}[\PP]))^{-1}[\psi]_{\bI}[\PP]\Bigg)\\
&
=D_{\bL}\Bigg((M_{\bL}([\psi]_{\bI}[\PP]))^{-1}M_{\bJ}([\psi]_{\bI}[\PP])(M_{\bJ}([\psi]_{\bI}[\PP]))^{-1}[\psi]_{\bI}[\PP]\Bigg)\\
&
=D_{\bL}\Bigg(\Big(M_{\bL}([\psi]_{\bI}\PP)\Big)^{-1}[\psi]_{\bI}\PP\Bigg).
\end{align*}‎
For right side of equation (\ref{glueaction}), we have
\begin{align*}
\textbf{A}_{\bQ}^{\bL}&\circ(g_{{\bQ},{\bI}})_{_T}(\psi)\\
&=\textbf{A}_{\bQ}^{\bL}\Bigg(D_{\bQ}\bigg((M_{\bQ}[\psi]_{\bI})^{-1}[\psi]_{\bI}\bigg)\Bigg)\\
&  =D_{\bL}\Bigg(\Big[M_{\bL}\bigg((M_{\bQ}[\psi]_{\bI})^{-1}[\psi]_{\bI}[\PP]\bigg)\Big]^{-1}(M_{\bQ}[\psi]_{\bI})^{-1}[\psi]_{\bI}[\PP]\Bigg)\\
&
=D_{\bL}\Bigg(\Big[(M_{\bQ}[\psi]_{\bI})^{-1}M_{\bL}([\psi]_{\bI}[\PP])\Big]^{-1}(M_{\bQ}[\psi]_{\bI})^{-1}[\psi]_{\bI}[\PP]\Bigg)\\
&
=D_{\bL}\Bigg(\Big(M_{\bL}([\psi]_{\bI}[\PP])\Big)^{-1}(M_{\bQ}[\psi]_{\bI})(M_{\bQ}[\psi]_{\bI})^{-1}[\psi]_{\bI}[\PP]\Bigg)\\
&
=D_{\bL}\Bigg(\Big(M_{\bL}([\psi]_{\bI}\PP)\Big)^{-1}[\psi]_{\bI}\PP\Bigg).
\end{align*}
This shows that the above diagram commutes.
\end{proof}
Therefore $GL(\bm)$ acts on $G_{\bk}(\bm)$ with action $a$.
Now it is needed to show that this action is transitive. 
\begin{theorem}
$GL(\bm)$ acts on $G_{\bk}(\bm)$ transitively.
\end{theorem}
\begin{proof}
By proposition \ref{Transitive}, it is sufficient to show that the map $$(a_p)_{\mathbb{R}^{\br^\prime}}:GL(\bm)(\mathbb{R}^{\br^\prime}) \rightarrow G_{\bk}(\bm)(\mathbb{R}^{\br^\prime}),$$ is surjective,
where $\br=(r_0,r_1,\ldots,r_{_\qn})$ is dimension of $GL(\bm)$ and $\br^\prime=(0,r_1,\ldots,r_{_\qn}).$
Let $$p=(p_0,p_1,\ldots,p_\qn)\in\overline{U}_{\bI}\subset 
G_{k_0}(m_0)\times G_{k_1}(m_1)\times \ldots \times G_{k_\qn(m_\qn)}$$
be an element and $\bar{p}_0,\bar{p}_1,\ldots, \bar{p}_{_\qn}$ be the matrices corresponding to subspaces $p_0,p_1,\ldots,p_\qn$ respectively. As an element of $G_{\bk}(\bm)(T)$, one may represent $\hat{p}_{_T}$, as follows
$$\hat{p}_{_T}=\left[\begin{array}{c|c|c|c} 
\overline{p}_0 & 0 & \ldots & 0\\
\hline 
0 & \overline{p}_1 & \ldots & 0\\
\hline
0 & 0 & \ldots & 0\\
\hline
0 & 0 & \ldots & \overline{p}_{_\qn}
\end{array}
\right]$$
where $T$ is an arbitrary $\Zn-$supermanifold.
For surjectivity, let 
$$W=\left[\begin{array}{c|c|c|c} 
W_{00} & W_{01} & \ldots & W_{0\qn}\\
\hline
W_{10} & W_{11} & \ldots & W_{1\qn}\\
\hline 
\ldots & \ldots & \ldots & \ldots\\
\hline
W_{\qn0} & W_{\qn1} & \ldots & W_{\qn\qn}
\end{array}
\right]\in U_{\bJ}(\mathbb{R}^{\br^\prime}),$$
 be an arbitrary element.
According to \eqref{apag}, we have to show that there exists an element $V\in GL(\bm)(\mathbb{R}^{\br^\prime})$ such that $\hat{p}_{_T}V=W$. Since 
 the Lie group $GL(m_i)$ acts on manifold $G_{k_i}(m_i)$ transitively, then there exists an invertible matrix $H_{ii}\in GL(m_i)$ such that $\bar{p}_iH_{ii}=W_{ii}$. In addition, the equations $\bar{p}_iZ=W_{ij}$ have solutions since $rank(\bar{p}_i)=k_i$. Let $H_{ij}$
be solutions of these equations respectively. Clearly, One can see $$V=\left[\begin{array}{c|c|c|c}
H_{00} & H_{01} & \ldots & H_{0\qn}\\
\hline H_{10} & H_{11} & \ldots & H_{1\qn}\\
\hline \ldots & \ldots & \ldots & \ldots\\
\hline H_{\qn0} & H_{\qn1} & \ldots & H_{\qn\qn}
\end{array}
\right]_{\bm\times \bm}$$
satisfy in  the equation $\hat{p}_{_T}V=W$. So $(a_p)_{\mathbb{R}^{\br^\prime}}$ is surjective.
By Proposition \ref{Transitive}, $GL(\bm)$ acts on $G_{\bk}(\bm)$ transitively.
\end{proof}
Thus according to Proposition \ref{equivariant}, $G_{\bk}(\bm)$ is a homogeneous $\Zn-$superspace.

\newpage
\providecommand{\bysame}{\leavevmode\hbox to3em{\hrulefill}\thinspace}
\providecommand{\MR}{\relax\ifhmode\unskip\space\fi MR }

\providecommand{\href}[2]{#2}

\end{document}